\documentclass{amsart}
\usepackage{color}
\usepackage{mathdots}
\usepackage{hyperref}

\makeatletter
\@namedef{subjclassname@2010}{%
  \textup{2010} Mathematics Subject Classification}
\makeatother

\usepackage{amssymb,amsmath,array,lineno,comment,enumerate,microtype}

\usepackage{graphicx,subfigure}
\usepackage{cite}

\def\Q{{\mathbb Q}}

\def\D{{\mathcal D}}

\newtheorem{theorem}{Theorem}[section]
\newtheorem{lemma}[theorem]{Lemma}
\newtheorem{prop}[theorem]{Proposition}
\newtheorem{cor}[theorem]{Corollary}

\begin{document}
\title{Towards the (ir)rationality of values of Dirichlet series}

\author{Michael Coons}
\address{School of Math.~and Phys.~Sciences\\
University of Newcastle\\
Callaghan\\
Australia}
\email{Michael.Coons@newcastle.edu.au}

\author{Daniel Sutherland}
\address{School of Math.~and Phys.~Sciences\\
University of Newcastle\\
Callaghan\\
Australia}
\email{Daniel.Sutherland@newcastle.edu.au}


\date{\today}

\begin{abstract}
We show that if $F(s)$ is a nondegenerate ordinary Dirichlet series with nonnegative coefficients and $F(k)$ is a rational number for all large enough positive integers $k$, then the denominators of those rational numbers are unbounded. In particular, our result holds for the Riemann zeta function over any arithmetic progression. These results are derived via upper bounds on associated Hankel determinants.
\end{abstract}

\maketitle

\vspace{-.5cm}
\section{Introduction}

The values of the Riemann zeta function at positive even integers were determined by Euler nearly 300 years ago. Together with Lindemann's proof of the transcendence of $\pi$, for over 130 years we have known that $\zeta(2n)$ is transcendental for $n\geqslant 1$. The complementary irrationality results for zeta values at odd integers has been of great importance in the mathematical community for some time. Comparitively recently (only 35 years ago), Ap\'ery \cite{Apery} showed that $\zeta(3)$ was irrational, though unfortunately, his proof does not extend to other odd zeta values; see also Beukers \cite{Beukers}. The story ends here for irrationality of specific zeta values, though one can say more with less specified outcomes. Rivoal \cite{Rivoal} has shown that infinitely many odd zeta values are irrational, and Zudilin \cite{Zudilin} has shown that one of $\zeta(5)$, $\zeta(7)$, $\zeta(9)$, or $\zeta(11)$ is irrational. 

Similar to Rivoal's result, this paper concerns (ir)rationality of ordinary Dirichlet series, with particular attention to zeta values over arithmetic progressions. Concerning the specific case of Riemann's zeta function, in this paper we contribute the following result.

\begin{theorem}\label{mainzeta} Let $a$ and $b$ be positive integers and let $\zeta_{a,b}(k)=\sum_{m\geqslant 1}{m^{-(ak+b)}}$. Suppose for some $R>0$ that $\zeta_{a,b}(k)\in\Q$ for all integers $k\geqslant R$. For each $k\geqslant R$, define the positive pair of coprime integers $p_k$ and $q_k$ by ${p_k}/{q_k}=\zeta_{a,b}(k)$. Then the sequence $\{q_k\}_{k\geqslant R}$ is unbounded.
\end{theorem}

An analogous result to Theorem \ref{mainzeta} is true replacing $\zeta_{a,b}(s)$ with any nondegenerate ordinary Dirichlet series with nonnegative coefficients. We prove the following generalisation in this paper. 

\begin{theorem}\label{thm:main}
Let $F(s)=\sum_{n\geqslant 1} f(n)n^{-s}$ be a nondegenerate ordinary Dirichlet series with $f(n)\geqslant 0$ for all $n$, which is convergent for $\Re(s)\geqslant s_0$. Suppose that $F(k)\in\Q$ for all integers $k\geqslant R\geqslant s_0-2$, and for each $k\geqslant R$, define the positive pair of coprime integers $p_k$ and $q_k$ by ${p_k}/{q_k}=F(k)$ and set
	\[
		\D_m[F] = {\rm lcm}(q_{R+2},q_{R+3},\ldots,q_{R+m}).
	\]
Then the common denominator $\D_m[F]$ grows at least exponentially in $m$.
\end{theorem}

In the case of $\zeta_{a,b}(s)$, if one keeps track of constants, Theorem \ref{thm:main} implies that for each $\varepsilon\in(0,1)$ that $\D_m[\zeta_{a,b}]>(2-\varepsilon)^m$ for large enough $m$ depending on $\varepsilon$. This in turn implies that the $q_k$, as defined in Theorem \ref{mainzeta}, satisfy $$\max_{k\leqslant m}\{q_k\}>m\log(2-\varepsilon)$$ for any $\varepsilon\in(0,1)$ and $m$ large enough.

Theorem \ref{mainzeta} follows immediately from Theorem \ref{thm:main}; we prove Theorem \ref{thm:main} via a result on Hankel determinants. For a given sequence of real numbers $\{h(k)\}_{k\geqslant 2}$ and integers $n$ and $r$, we define the Hankel determinant of size $n$ starting at offset $r$ by
\[
	H_n^{(r)}[h] = \det_{1\leqslant i,j\leqslant n}\Big(h(i+j+r)\Big).
\]

The growth of these determinants plays an important role in investigating the (ir)rationality of zeta values. Theorems \ref{mainzeta} and \ref{thm:main} are easily deduced from analogous statements on the decay of Hankel determinants. Indeed, Theorem \ref{thm:main} is a consequence of the following result on the growth of Hankel determinants of nondegenerate ordinary Dirichlet series with nonnegative coefficients.

\begin{theorem}\label{main}
Let $F(s)=\sum_{n\geqslant 1} f(n)n^{-s}$ be a nondegenerate ordinary Dirichlet series with $f(n)\geqslant 0$ for all $n$, which is convergent for $\Re(s)\geqslant s_0$. Then $H_n^{(r)}[F]>0$ for all $n\geqslant 1$ and $r\geqslant s_0-2$, and there is a $c>0$ such that
	\[
		\log H_n^{(r)}[F] < -c n^2,
	\]
for all sufficiently large integers $n$ and $r$. 
\end{theorem}

We note that Theorem \ref{main} is valid for $r$ large enough, where ``$r$ large enough'' can in reality be ``$r$ very small.'' In the case of the Riemann zeta function, the theorem holds for all $r\geqslant 0$. In fact, our result implies that $$\log H_n^{(0)}[\zeta]<-n^2\log(2-\varepsilon),$$   for any positive $\varepsilon$ close to zero and large enough $n$. 

Experimental work suggests that the asympotitic decay of the Hankel determinants of the Riemann zeta function for $r=0,1$ is a bit better than what Theorem~\ref{main} concludes. In a preprint, Monien \cite{Monien} has produced a heuristic, which suggests that $$\log H_n^{(0)}[\zeta]\sim\log H_n^{(1)}[\zeta]\sim -n^2\log(2n)=-n^2\log 2-n^2\log n.$$ Monien also found experimentally, that $$
-\frac{H_{n-1}^{(0)}[\zeta]H_{n}^{(1)}[\zeta]}{H_{n}^{(0)}[\zeta]H_{n+1}^{(1)}[\zeta]}  =  -\frac{1}{2n+1}+\frac{2}{(2n+1)^{2}}-\frac{7}{3}\frac{1}{(2n+1)^{3}}+\cdots,$$ and $$
-\frac{H_{n+1}^{(0)}[\zeta]H_{n-1}^{(1)}[\zeta]}{H_{n}^{(0)}[\zeta]H_{n}^{(1)}[\zeta]}  =  -\frac{1}{2n}-\frac{1}{(2n)^{2}}+\frac{2}{3}\frac{1}{(2n)^{3}}-\frac{6}{5}\frac{1}{(2n)^{4}}+\frac{56}{45}\frac{1}{(2n)^{5}}+\cdots.$$ Additionally, according to Monien, detailed numerical experiments by Zagier suggest that $$H_{n}^{(0)}[\zeta]=A^{(0)}\left(\frac{2n+1}{e\sqrt{e}}\right)^{-(n+\frac{1}{2})^{2}}\left(1+\frac{1}{24}\frac{1}{(2n+1)^{2}}-\frac{12319}{259200}\frac{1}{(2n+1)^{4}}+\ldots\right),$$ and $$
H_{n-1}^{(1)}[\zeta]=A^{(1)}\left(\frac{2n}{e\sqrt{e}}\right)^{-n^{2}+\frac{3}{4}}\left(1-\frac{17}{240}\frac{1}{(2n)^{2}}-\frac{199873}{7257600}\frac{1}{(2n)^{4}}-\ldots\right),$$ where $A^{(0)}\approx0.351466738331\ldots$ and $A^{(1)} =\frac{e^{9/8}}{\sqrt{6}}A^{(0)}.$

\section{Hankel Determinants of some specialised sequences}\label{Special}

This section contains two results, which are of paramount importance to our investigation. The first is Dodgson condensation---sometimes known as ``Lewis Carroll's identity.'' 

\begin{lemma}[Dodgson \cite{Dodgson}]\label{prop:rr}
	Let $n\geqslant 2$ and $r\geqslant 0$ be integers and $\{h(k)\}_{k\geqslant 0}$ be a sequence of real numbers. Then
	\[
		H_{n+1}^{({r})}[h]\cdot H_{n-1}^{(r+2)}[h]=H_n^{(r)}[h]\cdot H_n^{(r+2)}[h]-\left(H_n^{(r+1)}[h]\right)^2.
	\]
\end{lemma}

While the setting in Lemma \ref{prop:rr} is quite specific to Hankel determinants, Dodgson condensation is fully generalizable for evaluation of any determinant. Indeed, this is the way that Dodgson used it. For a good reference on Hankel determinants see P\'olya and Szeg\H{o} \cite[Part 7]{Polya}.

Our next result provides a way to bound the Hankel determinants of certain sequences by using bounds on the growth of those sequences and their consecutive ratios. 

\begin{prop}\label{prop:decreaseN} Let $K\geqslant 2$ be a fixed integer and $h(k)>0$ for $k\geqslant K$. Suppose that there exist positive functions $A(k)$, $B(k)$ and $\lambda(k)$, with $A(k)<B(k+1)$ for all $k\geqslant K$, such that
\begin{enumerate}
\item[(i)]	$A(k) \leqslant \frac{h(k+1)}{h(k)}<B(k)$ for all $k\geqslant K$, and 
\item[(ii)]	$h(k+2) < \lambda(k+2)\cdot\frac{B(k+1)}{B(k+1)-A(k)},$ for all $k\geqslant K$. 
\end{enumerate}
If 
$H_n^{(r)}[h]>0$ for all $n\geqslant 1$ and for $r\geqslant K-2$, then  
	\[
		H_n^{(r)}[h] < h(2+r)\prod_{k=2}^n\lambda(2k+r),
	\]
for all $n\geqslant 2$ and for $r\geqslant K-2$.
\end{prop}

\begin{proof}
	Fix $r\geqslant K-2$. By definition, $H_2^{(r)}[h]=h(2+r)h(4+r)-h(3+r)^2$ and $H_1^{(r)}[h]=h(2+r)$. Thus using the positivity of $h(k)$ along with assumptions $(i)$ and $(ii)$, we have
	\begin{align*}
		H_2^{(r)}[h] &= h(2+r)h(4+r)\left(1-\frac{h(3+r)}{h(4+r)}\frac{h(3+r)}{h(2+r)}\right) \\
		 &< h(2+r)h(4+r)\left(1-\frac{A(2+r)}{B(3+r)}\right)\\
		&< H_1^{(r)}[h]\cdot\lambda(4+r).
	\end{align*}
	
	Suppose that $n\geqslant 2$. By Lemma \ref{prop:rr} and the positivity of $H_n^{(r)}[h]$, we have $$ H_{n+1}^{({r})}[h]=\frac{H_n^{(r)}[h]\cdot H_n^{(r+2)}[h]-\left(H_n^{(r+1)}[h]\right)^2}{H_{n-1}^{(r+2)}[h]}<H_n^{(r)}[h]\cdot \frac{H_n^{(r+2)}[h]}{H_{n-1}^{(r+2)}[h]}.$$
Combining this with the analogous result for $H_n^{(r+2)}[h]$ gives $$H_{n+1}^{({r})}[h]<H_n^{(r)}[h]\cdot \frac{H_{n-1}^{(r+2)}[h]\cdot \frac{H_{n-1}^{(r+4)}[h]}{H_{n-2}^{(r+4)}[h]}}{H_{n-1}^{(r+2)}[h]}=H_n^{(r)}[h]\cdot\frac{H_{n-1}^{(r+4)}[h]}{H_{n-2}^{(r+4)}[h]}.$$
We continue this process to keep reducing $n$ (while increasing $r$) until we can apply the $n=1$ case to give
	\[
		H_{n+1}^{(r)}[h] < H_{n}^{(r)}[h]\cdot \frac{H_2^{(r+2(n-1))}[h]}{H_1^{(r+2(n-1))}[h]} < H_n^{(r)}[h]\cdot\lambda(2(n+1)+r).
	\]
	
Now let $n\geqslant 3$. The previous inequality gives $H_{n}^{(r)}[h]<H_{n-1}^{(r)}[h]\cdot \lambda(2n+r)$. Repeated application of this inequality gives
$$H_{n}^{(r)}[h] < H_{1}^{(r)}[h]\cdot \prod_{k=2}^n\lambda(2k+r)=h(2+r)\prod_{k=2}^n\lambda(2k+r),$$ which is the desired result.
\end{proof}

In what follows we will only be interested in decreasing sequences $\{h(k)\}_{k\geqslant2}$. We note that, if additionally $h(k)$ is bounded for all $k$, say by $M$, then condition {\em (ii)} of Proposition \ref{prop:decreaseN} will be satisfied by defining $B(k)=B$ for some constant $B>0$ and $\lambda(k)=(B-A(k-2))M/B.$ This will be the setting in the following section, though Proposition \ref{prop:decreaseN} holds for increasing functions as well.

By way of example, consider the sequence $\{h(k)\}_{k\geqslant2}$ where $h(k)=(k-2)!$. Denote by $n\$$ the superfactorial of $n$; that is, $$n\$=\prod_{k=1}^n k!.$$ Strehl has shown that $H_n^{(r)}[h]>0$ for all $n\geqslant1$ and $r\geqslant0$; in fact, exact values for these determinants are known---see Radoux \cite{Radoux} for details. To gain an upper bound, we can apply Proposition \ref{prop:decreaseN} with $K=2$, $A(k)=k-1$, $B(k)=k$ and $\lambda(k)=2(k-2)!$ to obtain $$H_n^{(r)}[h]<2^{n-1}\prod_{k=1}^n(2k+r-2)!=2^{n-1}\cdot r!(r+2)!(r+4)!\cdots(r+2n-2)!,$$ for all $n\geqslant1$ and $r\geqslant0$.

\section{Hankel determinants of ordinary Dirichlet series}\label{Ordinary}

In this section, we consider general ordinary Dirichlet series with nonnegative coefficients. There is a natural split between what can be considered a degenerate case, and a nondegenerate case. As the degenerate case, we consider a Dirichlet series having only finitely many nonzero coefficients, and as the nondegenerate case, we consider Dirichlet series having infinitely many nonzero coefficients.

We use the following result for the nonvanishing of values of Hankel Determinants.

\begin{theorem}[Monien \cite{Monien}]\label{thm:M} Let $F(s)=\sum_{n\geqslant 1} f(n)n^{-s}$ be an ordinary Dirichlet series, which is convergent for $\Re(s)\geqslant s_0$. Then for any $n\geqslant 1$ and $r\geqslant s_0-2$, we have $$H_{n}^{(r)}[F]=\frac{1}{n!}\sum_{m_{1},m_{2},\ldots m_{n}=1}^{\infty}\prod_{i=1}^{n}\frac{f(m_{i})}{m_{i}^{2n+r}}\prod_{i<j}(m_{i}-m_{j})^{2}.$$
\end{theorem}

While Theorem \ref{thm:M} holds in great generality, a classical result due to Kronecker (see also P\'olya and Szeg\H{o} \cite[Part 7]{Polya}) provides a nice dichotomy.

\begin{theorem}[Kronecker] Let $\{a(n)\}_{n\geqslant 0}$ be a sequence of real numbers. Then the function $\sum_{n\geqslant 0}a(n)x^n$ is rational if and only if finitely many of the determinants $H_n^{(r)}[a]$ are nonzero.
\end{theorem}

This leads to the following result.

\begin{lemma}\label{nonzero} Let $F(s)=\sum_{n\geqslant 1} f(n)n^{-s}$ be an ordinary Dirichlet series with $f(n)\geqslant 0$ for all $n$, which is convergent for $\Re(s)\geqslant s_0$. Then 
\begin{enumerate}
\item[(i)] if there are only finitely many $f(n)\neq 0$, then $H_n^{(r)}[F]=0$ for sufficiently large $n$ and any $r\geqslant s_0-2$;
\item[(ii)] if there are infinitely many $f(n)\neq 0$, then $H_n^{(r)}[F]>0$ for any $n\geqslant 1$ and $r\geqslant s_0-2$.
\end{enumerate}
\end{lemma}

\begin{proof} Theorem \ref{thm:M} implies both {\em (i)} and {\em (ii)} immediately. Part {\em (i)} also follows from Kronecker's theorem.
\end{proof}

To further examine the Hankel determinants in the nondegenerate case, we require a lower bound on the ratio of successive values of $F(s)$.

\begin{lemma}\label{MandN} Let $F(s)=\sum_{n\geqslant 1} f(n)n^{-s}$ be a nondegenerate ordinary Dirichlet series with $f(n)\geqslant 0$ for all $n$, which is convergent for $\Re(s)\geqslant s_0$. If $N<M$ are the two minimal indices of the nonzero coefficients of $F(s)$, then the following hold:
\begin{enumerate} 
\item[(i)] if $N=1$, then $$\lim_{s\to\infty}\frac{M^{s+1}}{(M-1)f(M)}\left(1-\frac{F(s+1)}{F(s)}\right)=1;$$
\item[(ii)] if $N>1$, then $$\lim_{s\to\infty}\frac{f(N)N^{(\alpha-1)s}}{f(M)(1-N^{1-\alpha})}\left(1-N\cdot \frac{F(s+1)}{F(s)}\right)=1,$$ where $\alpha=\log M/\log N.$
\end{enumerate} 
\end{lemma}

\begin{proof} We treat the two cases, $N=1$ and $N>1$, separately.

{\sc Case $N=1$.} Note that by definition, we have 
$$F(s)=1+f(M)\cdot M^{-s}+O\left(M^{-\delta s}\right),$$ where $\delta=\log(M+1)/\log M>1.$
Thus 
$$\frac{F(s+1)}{F(s)}=\frac{1+\frac{f(M)}{M}\cdot M^{-s}+O\left(M^{-\delta s}\right)}{1+f(M)\cdot M^{-s}+O\left(M^{-\delta s}\right)}=1-(M-1)f(M) M^{-(s+1)}+O\left(M^{-\delta s}\right),$$ and so $$\frac{M^{s+1}}{(M-1)f(M)}\left(1-\frac{F(s+1)}{F(s)}\right)=1+O\left(\frac{M^{s+1}}{M^{\delta s}}\right).$$ Since $\delta>1$, taking $s$ to infinity gives the required result.

{\sc Case $N>1$.} By definition, we have 
\begin{align*} F(s)&=f(N)\cdot N^{-s}+f(M)\cdot N^{-\alpha s}+O\left(N^{-\beta s}\right)\\ &=f(N)\cdot N^{-s}\left(1+\frac{f(M)}{f(N)}\cdot N^{(1-\alpha) s}+O\left(N^{(1-\beta) s}\right)\right),\end{align*} where $\alpha=\log M/\log N$ and $\beta=\log(M+1)/\log N$. Note that $\beta>\alpha>1$.
Thus 
\begin{align*} \frac{F(s+1)}{F(s)}&=\frac{f(N)\cdot N^{-s-1}\left(1+\frac{f(M)}{f(N)}\cdot N^{(1-\alpha)(s+1)}+O\left(N^{(1-\beta)s}\right)\right)}{f(N)\cdot N^{-s}\left(1+\frac{f(M)}{f(N)}\cdot N^{(1-\alpha) s}+O\left(N^{(1-\beta) s}\right)\right)}\\
&=\frac{1}{N}\left(1-\frac{f(M)}{f(N)}\left(1-N^{1-\alpha}\right)N^{(1-\alpha)s}+o\left(N^{(1-\alpha)s}\right)\right),
\end{align*} and so $$\frac{f(N)N^{(\alpha-1)s}}{f(M)(1-N^{1-\alpha})}\left(1-N\cdot \frac{F(s+1)}{F(s)}\right)=1+o\left(1\right).$$ This completes the proof of the lemma.
\end{proof}

This lemma provides the following immediate corollary.

\begin{cor}\label{govergasym} Let $F(s)=\sum_{n\geqslant 1} f(n)n^{-s}$ be a nondegenerate ordinary Dirichlet series with $f(n)\geqslant 0$ for all $n$, which is convergent for $\Re(s)\geqslant s_0$. If $N<M$ are the two minimal indices of the nonzero coefficients of $F(s)$, then the following hold:
\begin{enumerate} 
\item[(i)] if $N=1$, then there are positive constants $c_1>0$ and $K_1\geqslant s_0$ such that for all $s\geqslant K_1\geqslant s_0$, we have $$0<1-\frac{c_1}{M^{s}} \leqslant \frac{F(s+1)}{F(s)}< 1;$$
\item[(ii)] if $N>1$, then there are positive constants $c_2>0$ and $K_2\geqslant s_0$ such that for all $s\geqslant K_2\geqslant s_0$, we have $$0<\frac{1}{N}-\frac{c_2}{N^{(\alpha -1)s}} \leqslant \frac{F(s+1)}{F(s)}< \frac{1}{N},$$ where $\alpha=\log M/\log N.$
\end{enumerate} 
\end{cor}

We are now in a position to prove the Theorem \ref{main}.

\begin{proof}[Proof of Theorem \ref{main}]  As in the proof of Lemma \ref{MandN}, we split this proof into the two cases, $N=1$ and $N>1$, where $N$ is the minimal index of nonzero coefficients of $F(s)$.

{\sc Case $N=1$.} 
Let $c_1,K_1>0$ be such that the conclusion of Corollary \ref{govergasym}{\em (i)} holds, and for all $k$, set $$A(k)=1-\frac{c_1}{M^{k}}.$$ Then we have $A(k)\leqslant F(k+1)/F(k)<1$ and $A(k)$ is a positive function for $k\geqslant K_1$. Also, $$\frac{1}{1-A(k-2)}=\frac{M^{k-2}}{c_1}.$$ 

Set $$\lambda(k)=2\cdot F(s_0)\cdot\big(1-A(k-2)\big)=2\cdot F(s_0)\cdot\frac{c_1}{M^{k-2}}.$$ Then  
$$\lambda(k)\cdot \frac{1}{1-A(k-2)} =2\cdot F(s_0)> F(k),$$ since $F(s)$ is a monotonically decreasing function of $s\geqslant s_0$.

We may now apply Proposition \ref{prop:decreaseN}, with $K=K_1$, so that 
		$$H_n^{(r)}[F] < F(2+r)\prod_{k=2}^n\lambda(2k+r)\leqslant F(s_0)\prod_{k=2}^n\lambda(2k),$$
for all integers $n\geqslant 2$ and $r\geqslant K$. Thus \begin{align*}
\log H_n^{(r)}[F] &< \log\left(F(s_0)\prod_{k=2}^n\lambda(2k)\right) \\
&=\log F(s_0)+\sum_{k=2}^n \log \lambda(2k)\\
&=\log F(s_0)+\sum_{k=2}^n \log \left(2\cdot F(s_0)\cdot \frac{c_1}{M^{2k-2}}\right)\\
&=\log F(s_0)+(n-1)\log(2 F(s_0) c_1 M^2) -2\log M\cdot\frac{n(n+1)-2}{2}\\
&\sim -n^2 \log M.
\end{align*} This completes the $N=1$ case.

{\sc Case $N>1$.}
Let $c_2,K_2>0$ be such that conclusion of Corollary \ref{govergasym}{\em (i)} holds, and for all $k$, set $$A(k)=\frac{1}{N}-\frac{c_2}{N^{(\alpha -1)k}}.$$ Then we have $A(k)\leqslant F(k+1)/F(k)<1/N$ and $A(k)$ is a positive function for $k\geqslant K_2$. 

Set $$\lambda(k)=2N\cdot F(s_0)\cdot\big(1/N-A(k-2)\big)=2N\cdot F(s_0)\cdot\frac{c_2}{N^{(\alpha-1)(k-2)}}.$$ Then  
$$\lambda(k)\cdot \frac{1/N}{1/N-A(k-2)} =2\cdot F(s_0)> F(k),$$ since $F(s)$ is a monotonically decreasing function of $s\geqslant s_0$.

We may now apply Proposition \ref{prop:decreaseN}, with $K=K_2$, so that 
		$$H_n^{(r)}[F] < F(2+r)\prod_{k=2}^n\lambda(2k+r)\leqslant F(s_0)\prod_{k=2}^n\lambda(2k),$$
for all integers $n\geqslant 2$ and $r\geqslant K$. Thus \begin{align*}
\log H_n^{(r)}[F] &< \log\left(F(s_0)\prod_{k=2}^n\lambda(2k)\right) \\
&=\log F(s_0)+\sum_{k=2}^n \log \lambda(2k)\\
&=\log F(s_0)+\sum_{k=2}^n \log \left(2N\cdot F(s_0)\cdot\frac{c_2}{N^{(\alpha-1)(k-2)}}\right)\\
&=\log F(s_0)+(n-1)\log(2N F(s_0) c_2 N^{2(\alpha-1)})\\
&\qquad\qquad -2(\alpha-1)\log N\cdot\frac{n(n+1)-2}{2}\\
&\sim -n^2\cdot 2(\alpha-1)\log N.
\end{align*} Recalling that $\alpha>1$ completes the $N>1$ case, and with that, the proof of the theorem.
\end{proof}

\begin{proof}[Proof of Theorem \ref{thm:main}]
To see how Theorem \ref{thm:main} is deduced from Theorem \ref{main}, suppose that the conclusion of Theorem \ref{main} holds; that is, let $N_0,R>0$ be such that for $n\geqslant N_0$ and $r\geqslant R$, we have $\log H_n^{(r)}[F] < -c n^2$. Now suppose that $F(k)$ is rational for all integers $k\geqslant R+2$.  As in the statement of Theorem \ref{thm:main}, write ${p_k}/{q_k}=F(k)$ and set
\[
	\D_m[F] = {\rm lcm}(q_{R+2},q_{R+3},\ldots,q_{R+m}).
\]
Then, using the positivity of $H_n^{(r)}[F]$ given by Theorem \ref{main}, we have
\[
	\big(\D_{2n}[F]\big)^n\cdot H_n^{(r)}[F]\in\mathbb{N}, 
\]
so that 
\begin{equation*}\label{g0}
	\log H_n^{(r)}[F]+n\log \D_{2n}[F] \geqslant 0,
\end{equation*}
for all integers $n\geqslant 1$. If the sequence $\D_m[F]$ satisfied $\D_m\leqslant C^m$ for some $C\in(1,e^c)$ where $c>0$ is as given by Theorem \ref{main} for large enough $m$, then using Theorem~\ref{main}, for large enough $n$ there is a $\delta>0$ such that $$0\leqslant \log H_n^{(r)}[F]+n\log \D_{2n}[F] =\log H_n^{(r)}[F]+(c-\delta)n^2<-\delta n^2,$$ which is a contradiction for $n$ large enough.
\end{proof}

\section{Concluding remarks}

Note that repeated use of Dodgson condensation gives \begin{align*}
H_{n+1}^{(r)}[h]&=H_{n}^{(r)}[h]\cdot \frac{H_{n}^{(r+2)}[h]}{H_{n-1}^{(r+2)}[h]}\cdot\left(1-\frac{\left(H_{n}^{(r+1)}[h]\right)^2}{H_{n}^{(r)}[h]\cdot H_{n}^{(r+2)}[h]}\right)\\
&=H_{n}^{(r)}[h]\cdot \frac{H_{2}^{(r+2(n-1))}[h]}{H_{1}^{(r+2(n-1))}[h]}\cdot\prod_{j=0}^{n-2}\left(1-\frac{\left(H_{n-j}^{(r+(j+1))}[h]\right)^2}{H_{n-j}^{(r)}[h]\cdot H_{n-j}^{(r+2j)}[h]}\right),
\end{align*} so, replacing $n+1$ with $n$, we have \begin{align*}H_{n}^{(r)}[h]
&=H_{n-1}^{(r)}[h]\cdot \frac{H_{2}^{(r+2(n-2))}[h]}{H_{1}^{(r+2(n-2))}[h]}\cdot\prod_{j=0}^{n-3}\left(1-\frac{\left(H_{n-j-1}^{(r+(j+1))}[h]\right)^2}{H_{n-j-1}^{(r)}[h]\cdot H_{n-j-1}^{(r+2j)}[h]}\right)\\
&=h(2+r)\cdot \left(\prod_{i=2}^n\frac{H_{2}^{(r+2(i-2))}[h]}{H_{1}^{(r+2(i-2))}[h]}\right)\cdot\prod_{i=2}^n\prod_{j=0}^{i-3}\left(1-\frac{\left(H_{i-j-1}^{(r+(j+1))}[h]\right)^2}{H_{i-j-1}^{(r)}[h]\cdot H_{i-j-1}^{(r+2j)}[h]}\right)\\
&=h(2+r)\cdot \left(\prod_{i=2}^n\frac{H_{2}^{(r+2(i-2))}[h]}{h(r+2(i-1))}\right)\cdot\prod_{i=2}^n\prod_{j=0}^{i-3}\left(1-\frac{\left(H_{i-j-1}^{(r+(j+1))}[h]\right)^2}{H_{i-j-1}^{(r)}[h]\cdot H_{i-j-1}^{(r+2j)}[h]}\right).
\end{align*} In the case of a nondegenerate ordinary Dirichlet series $F(s)$ with nonnegative coefficients with $r\geqslant s_0-2$, using asymptotics for $F(s)$, gives some $c>0$ such that $$\log\left(\prod_{i=2}^n\frac{H_{2}^{(r+2(i-2))}[F]}{F(r+2(i-1))}\right)\sim -cn^2.$$ As our method is based off using this term with a constant upper bound for the double product above, Theorem \ref{main} cannot be improved using our method. Getting upper bounds on the double product involves obtaining lower bounds on $\log H_n^{(r)}[F]$; we hope to address this in future work.\\

\noindent\textbf{Acknowledgements.} We thank Wadim Zudilin for valuable remarks on an earlier version of this paper.

\providecommand{\bysame}{\leavevmode\hbox to3em{\hrulefill}\thinspace}
\providecommand{\MR}{\relax\ifhmode\unskip\space\fi MR }
\providecommand{\MRhref}[2]{%
  \href{http://www.ams.org/mathscinet-getitem?mr=#1}{#2}
}
\providecommand{\href}[2]{#2}


\begin{thebibliography}{1}

\bibitem{Apery}
Roger Ap\'ery, \emph{Irrationalit\'e de $\zeta(2)$ et $\zeta(3)$}, Ast\'erisque
  \textbf{61} (1979), 11--13.

\bibitem{Beukers}
F.~Beukers, \emph{A note on the irrationality of {$\zeta (2)$} and {$\zeta
  (3)$}}, Bull. London Math. Soc. \textbf{11} (1979), no.~3, 268--272.
  \MR{554391 (81j:10045)}

\bibitem{Dodgson}
C.~L. Dodgson, \emph{Condensation of determinants, being a new and brief method
  for computing their arithmetical values}, Proceedings of the Royal Society of
  London \textbf{15} (1866), pp. 150--155 (English).

\bibitem{Monien}
H.~Monien, \emph{{Hankel determinants of Dirichlet series}}, ArXiv e-prints
  (2009).

\bibitem{Polya}
G.~P{\'o}lya and G.~Szeg{\H{o}}, \emph{Problems and theorems in analysis ii},
  Classics in Mathematics, Springer-Verlag, Berlin, 1998, Translated from the
  German by C. E. Billigheimer, Reprint of the 1976 English translation.
  \MR{1492448}

\bibitem{Radoux}
C.~Radoux, \emph{D\'eterminants de hankel et th\'eor\`eme de sylvester},
  S\'eminaire Lotharingien de Combinatoire \textbf{28} (1992), 115--122.

\bibitem{Rivoal}
Tanguy Rivoal, \emph{La fonction z\^eta de {R}iemann prend une infinit\'e de
  valeurs irrationnelles aux entiers impairs}, C. R. Acad. Sci. Paris S\'er. I
  Math. \textbf{331} (2000), no.~4, 267--270. \MR{1787183 (2001k:11138)}

\bibitem{Zudilin}
Wadim Zudilin, \emph{Arithmetic of linear forms involving odd zeta values}, J.
  Th\'eor. Nombres Bordeaux \textbf{16} (2004), no.~1, 251--291. \MR{2145585
  (2006j:11102)}

\end{thebibliography}
\end{document}